\documentclass[10pt,a4paper,reqno]{amsart}

\usepackage{amsmath}
\usepackage{amsfonts}
\usepackage{amssymb}
\usepackage{amsthm}
\usepackage{mathrsfs}
\usepackage[shortlabels]{enumitem}
\usepackage{graphicx}
\usepackage{color}
\usepackage{dsfont}
\usepackage[latin1]{inputenc}
\usepackage{MnSymbol}
\usepackage{enumitem}
\usepackage{extpfeil}
\usepackage{mathtools}
\usepackage{url}
\usepackage[margin=1.25in]{geometry}
\usepackage{fancyhdr}
\usepackage[all,cmtip]{xy}
\usepackage{hyperref}
\usepackage{tikz}
\usepackage{thm-restate}

\numberwithin{equation}{section}

\newcommand{\R}{{\mathds R}}

\newcommand{\Z}{{\mathds Z}}
\newcommand{\G}{{\mathcal{G}}}

\def\ker{{\rm{ker}}}

\def\dim{{\rm{dim}}}

\def\G{\Gamma}

\theoremstyle{plain}
\newtheorem{theorem}{Theorem}[section]

\newtheorem{corollary}[theorem]{Corollary}

\newtheorem{lemma}[theorem]{Lemma}

\newtheorem*{question*}{Question}

\theoremstyle{definition}
\newtheorem{remark}[theorem]{Remark}
\newtheorem*{acknowledgements*}{Acknowledgements}

\newtheorem{definition}[theorem]{Definition}
\newtheorem*{notation*}{Notation}
\newtheorem*{convention*}{Convention}
\newtheorem{construction}[theorem]{Construction}

\title{Finitely presented kernels of homomorphisms from hyperbolic groups onto free abelian groups}
\author{Robert Kropholler}
\address{Mathematics Institute, Zeeman Building, University of Warwick, Coventry CV4 7AL, United Kingdom}
\email{robertkropholler@gmail.com}
\author{Claudio Llosa Isenrich}
\address{Faculty of Mathematics, Karlsruhe Institute of Technology, Englerstr. 2, 76131 Karlsruhe, Germany}
\email{claudio.llosa@kit.edu}

\thanks{}

\begin{document}

\begin{abstract}
 For every $m\geq 2$ we produce an example of a non-hyperbolic finitely presented subgroup $H < G$ of a hyperbolic group $G$, which is the kernel of a surjective homomorphism $\phi: G\to \Z^m$. The examples we produce are of finiteness type $F_2$ and not $F_3$. 
\end{abstract}

\maketitle

\section{Introduction}

Finiteness properties of groups generalise the notions of being finitely generated and being finitely presented. We say that a group $G$ is of finiteness type $F_n$ for a non-negative integer $n$, if there is a classifying CW-complex $K(G,1)$ for $G$ with finitely many cells of dimension $\leq n$. Finite generation (resp. finite presentability) is equivalent to $F_1$ (resp. $F_2$). We call a group of type $F$ if it has a finite classifying space and of type $F_{\infty}$ if it is $F_n$ for all $n\geq 0$. The first examples of groups of finiteness type $F_n$ and not $F_{n+1}$ were constructed by Stallings \cite{Sta-63} for $n=2$ and by Bieri \cite{Bie-76} for all $n>2$. Since then it has been shown that groups with interesting finiteness properties can arise under many additional assumptions. For instance, groups of type $F_n$ and not $F_{n+1}$ can be simple \cite{SWZ19} or fundamental groups of smooth complex projective varieties \cite{DimPapSuc-09-II}. 

The class of hyperbolic groups has attracted much attention in geometric group theory. They were introduced by Gromov in the late 1980s \cite{Gro-87}. A group is \emph{hyperbolic} if it acts geometrically on a $\delta$-hyperbolic geodesic space, that is, a metric space all of whose geodesic triangles are $\delta$-thin. Hyperbolic groups share many strong properties. For instance, they are $F_n$ for all $n\geq 0$, have solvable word and conjugacy problem, satisfy a strong version of the Tits alternative and are characterised by the linearity of their Dehn functions. It is natural to ask which of these properties pass to all of their subgroups. 

For finiteness properties very recent work of Llosa Isenrich and Py shows that they do not pass to subgroups. For every $n\geq 0$ they produce an example of a subgroup of a hyperbolic group that is type $F_n$ and not $F_{n+1}$, answering an old question of Brady \cite{Bra-99}. These examples also provided new examples of finitely presented non-hyperbolic subgroups of hyperbolic groups that are not themselves hyperbolic. A variety of such examples had already been constructed before their work. The first fundamental result in this direction was due to Brady \cite{Bra-99}, who showed the existence of finitely presented subgroups of hyperbolic groups which are not themselves hyperbolic, by producing a subgroup of a hyperbolic group that is type $F_2$ and not $F_3$. Brady's examples were subsequently generalised by Lodha \cite{Lod-18} and by Kropholler \cite{Kro-21} who also produced subgroups of hyperbolic groups that are type $F_2$ and not $F_3$. Very recently Italiano, Martelli and Migliorini \cite{IMM-21,IMM-22} constructed the first example of a non-hyperbolic subgroup of a hyperbolic group that is type $F$, answering a well-known open problem. Building on their approach Llosa Isenrich, Martelli and Py \cite{LIMP-21} constructed the first subgroup of a hyperbolic group that is type $F_3$ and not $F_4$, before Llosa Isenrich and Py settled the problem for all $n\geq 0$ in \cite{LloPy-22}.

All of the aforementioned examples of finitely presented non-hyperbolic subgroups of hyperbolic groups have in common that they arise as kernels of homomorphisms from hyperbolic groups onto $\mathbb{Z}$. This raises the question if there are some fundamental reasons which imply that such subgroups always have to arise in this way. In particular, it is natural to ask if for every integer $m\geq 1$ there is a homomorphism of a hyperbolic group onto $\mathbb{Z}^m$ with finitely presented non-hyperbolic kernel. For non-positively curved groups a positive answer to this question can be readily deduced from a positive answer for $m=1$ by taking direct products. However, direct products of infinite hyperbolic groups are not hyperbolic, making this approach impossible for hyperbolic groups. Here we pursue a different approach to this problem based on a direct use of BNSR-invariants. 

BNSR-invariants were introduced by Bieri, Neumann, Strebel and Renz. They provide us with a tool for determining the finiteness properties of kernels of homomorphisms onto abelian groups via the connectivity properties of certain half-spaces defined by equivariant height maps to $\mathbb{R}$ on universal coverings of classifying spaces. In this work we explain how a method of Bux and Gonzalez for computing BNSR-invariants of right-angled Artin groups \cite{BuxGon-99} can be adapted to more general cube complexes. We then apply this method to certain 3-dimensional hyperbolic cube complexes to prove the following result.

\begin{theorem}\label{thm:Main}
 For every integer $m\geq 1$ there is a hyperbolic group $G$ and a surjective homomorphism $\phi: G\to \mathbb{Z}^m$ such that $\ker(\phi)$ is finitely presented, but not of type $F_{3}$.
\end{theorem}

As a direct consequence of Theorem \ref{thm:Main}, we obtain the following result.

\begin{corollary}\label{cor:Main}
 For every integer $m\geq 1$ there is a hyperbolic group $G$ and a surjective homomorphism $\phi: G\to \Z^m$ such that $\ker(\phi)$ is finitely presented and not hyperbolic.
\end{corollary}

It would be interesting to extend Theorem \ref{thm:Main} to also produce examples of non-hyperbolic subgroups of hyperbolic groups of type $F_k$ and not $F_{k+1}$ as kernels of homomorphisms onto $\Z^m$ for all integers $k> 2$ and all integers $m\geq 2$. In \cite{LloPy-23} Llosa Isenrich and Py use methods from complex geometry to give a positive answer to this question for $m=2$ and arbitrary $k$. To our knowledge the case $m>2$ and $k>2$ remains open.

To prove Theorem \ref{thm:Main}, we show that the methods developed by Bux and Gonzalez in \cite{BuxGon-99} can be applied to two classes of examples of 3-dimensional CAT(0) cube complexes. The first class consists of generalisations of the examples constructed by Brady in \cite{Bra-99} and the second class consists of generalisations of the examples constructed by Lodha in \cite{Lod-18}. Both of these generalisations arise as special cases of the examples introduced by Kropholler in \cite{Kro-21}. Here we will carefully explain the additional conditions which we require for each of the two classes and how the methods developed by Bux and Gonzalez \cite{BuxGon-99} can be phrased in a more general context, so that they can be applied to our examples.

\subsection*{Structure} In Section \ref{sec:BNSR-invariants} we introduce the necessary background on BNSR-invariants and their relation to finiteness properties of coabelian subgroups. In Section \ref{sec:BuxGon} we explain how the results of Bux and Gonzalez from \cite{BuxGon-99} extend to the more general context of height maps on CAT(0) cube complexes. In Section \ref{sec:Brady-type-examples} we introduce a generalisation of Brady's examples \cite{Bra-99} of hyperbolic groups with subgroups of type $F_2$ and not $F_3$ due to Kropholler \cite{Kro-21} and then explain how to prove Theorem \ref{thm:Main} using these examples. In Section \ref{sec:Lodha-type-examples} we explain how to prove Theorem \ref{thm:Main} using examples which generalise Lodha's example from \cite{Lod-18}.

\section{Finiteness properties and BNSR-invariants}\label{sec:BNSR-invariants}

Introduced by Bieri, Neumann, Strebel and Renz \cite{BNS-87,BieRen-88, Ren-88} the BNSR-invariants of a group $G$ describe the finiteness properties of kernels of homomorphism onto abelian groups. They are subsets of the character sphere
\[
	S(G):= \left({\rm Hom}(G,\mathbb{R})\setminus \left\{0\right\}\right) /\sim,
\]
where $\chi_1\sim \chi_2$, for $\chi_1,~\chi_2: G \to \mathbb{R}$, if $\chi_1=\lambda \chi_2$ for some $\lambda\in \mathbb{R}_{>0}$. For a subset $M\subseteq G$ we denote by
\[
	S(G,M):= \left\{[\chi]\in S(G) \mid \chi|_M\equiv 0\right\}\subseteq S(G)
\]
the subsphere consisting of characters that vanish on $M$.

There is also a (weaker) homological version of BNSR-invariants. For an abelian unital ring $R$ we say that a group $G$ is of type $FP_n(R)$ if there is a partial projective resolution of length $n$ of the trivial $RG$-module $R$
\[
P_n\to P_{n-1}\to \cdots \to P_1\to P_0\to R\to 0
\]
by finitely generated projective $RG$-modules $P_i$, $0\leq i \leq n$.

The (homotopical and homological) BNSR-invariants of a group $G$ of type $F_n$ are descending sequences of subsets
\[
	\Sigma^n(G)\subseteq \Sigma^{n-1}(G)\subseteq \cdots \subseteq \Sigma^1(G)\subseteq \Sigma^0(G)=S(G)
\]

\[
	\Sigma^n(G,R)\subseteq \Sigma^{n-1}(G,R)\subseteq \cdots \subseteq \Sigma^1(G,R)\subseteq \Sigma^0(G,R)=S(G).
\]
Roughly speaking $[\chi]\in \Sigma^k(G)$ if the positive half-space defined by a $\chi$-equivariant height map $\widetilde{K(G,1)}\to \mathbb{R}$ on the universal cover of a $K(G,1)$ is $(k-1)$-connected. The precise definition is a bit more technical and we will omit it here, since we will not require it; it can be found in \cite{Ren-88} or \cite[Appendix B]{BieStr-92} (see also \cite[Definition 12]{LloPy-22}). The main Theorem about BNSR-invariants that we will use describes their relation to finiteness properties of kernels of homomorphisms to free abelian groups:

\begin{theorem}[{Bieri, Neumann, Strebel and Renz \cite{BNS-87,BieRen-88, Ren-88}}]\label{thm:Sigma-BNSR}
	Let $G$ be a group of type $F_n$ and let $\chi \in {\rm Hom}(G,\mathbb{R})$ be a homomorphism to a free abelian group. Then 
	\[
		\ker(\chi) \mbox { is } F_n \mbox{ (resp. } FP_n(\mathbb{Z})\mbox{)} \Longleftrightarrow S(G,\ker(\chi))\subseteq \Sigma^n(G) \mbox{ (resp. } S(G,\ker(\chi))\subseteq \Sigma^n(G,\mathbb{Z}) \mbox{)}.
	\]
\end{theorem}

\section{BNSR-invariants of cubulable groups}\label{sec:BuxGon}

Let $X$ be a finite cube complex, let $G=\pi_1(X)$ be its fundamental group, and let $\pi : \widetilde{X}\to X$ be its universal covering. We assume that every $n$-cube is identified with the unit cube $\left[0,1\right]^n$. This equips $\widetilde{X}$ and $X$ with natural metrics, whose restrictions to cubes are Euclidean. The link $L_{v,\widetilde{X}}$ of a vertex $v\in \widetilde{X}$ (respectively $X$) is the intersection of $\widetilde{X}$ (resp. $X$) with a sphere of radius $\frac{1}{4}$. It carries a natural cellular structure induced by the cell structure on $\widetilde{X}$ (resp. $X$). Assume that $\widetilde{X}$ is CAT(0) (and thus that $X$ is a $K(G,1)$). By Gromov's Link condition this is equivalent to the assumption that the link of every vertex in $\widetilde{X}$ (or equivalently $X$) is a flag complex \cite[Chapter II.5]{BriHae-99}.

We call a continuous map $h: \widetilde{X}\to \mathbb{R}$ an {\em (affine) height map} for a character $\chi: G\to \mathbb{R}$ if its restriction to every cube is affine linear and $h$ is $G$-invariant with respect to $\chi$ and the natural $G$-action on $\widetilde{X}$.

We will denote edges in $\widetilde{X}$ by $\left[v_0,v_1\right]$, where $v_0$ and $v_1$ are its endpoints (they are distinct, since $\widetilde{X}$ is a CAT(0) cube complex). An affine height map $h$ equips edges in $\widetilde{X}$ with weights $\mu(\left[v_0,v_1\right]):= \left| h(v_0)-h(v_1)\right|$. Moreover, edges of non-zero weight can be equipped with a canonical orientation defined by the direction in which $h$ is increasing. 

By $G$-equivariance of $h$ the weight and orientiation of the edges in $\widetilde{X}$ are $G$-invariant and thus define weights and orientations on the edges of the finite cube complex $X$. In particular, 
\[
C_{min}:= \inf\left\{ |\mu(\left[v_0,v_1\right])|\mid \left[v_0,v_1\right] \mbox{ edge of non-zero weight in $\widetilde{X}$} \right\}>0
\] 
is attained in some edge.

In analogy to \cite{BuxGon-99} we will now use the orientation on edges to define the dead and living subcomplex of the links of vertices and then subsequently define ascending and descending links of dead subcomplexes of $\widetilde{X}$. To define them we will require the notion of a link in a simplicial complex. While related, this should not be confused with the link of a vertex in a cube complex, respectively the more general notion of the link of a subcomplex of a cube complex which we shall introduce below. We will be careful to choose our notation so that it is always clear which one of the three we are referring to.

For a simplicial complex $L$ we define the \emph{link of a simplex} $\sigma\in L$ as 
\[
Lk_L(\sigma):= \left\{\tau\in L\mid \tau \cap \sigma = \emptyset,~ \tau \cup \sigma \in L\right\}.
\] 

Let $L_{v,X}$ be the link of a vertex $v\in X$. The weights on edges of $X$ allow us to equip the vertices of $L_{v,X}$ with weights. We define the following subcomplexes of $L_{v,X}$:
\begin{itemize}
\item the \emph{dead link} $L_{v,X}^{\dagger}$ is the flag subcomplex spanned by all vertices of weight zero;
\item the \emph{ascending living link} $L_{v,X}^{\ast,\uparrow}$ is the flag subcomplex spanned by all vertices of non-zero weight corresponding to outgoing edges in $v$; and the \emph{descending living link} $L_{v,X}^{\ast,\downarrow}$ is the flag subcomplex spanned by all vertices of non-zero weight corresponding to incoming edges at $v$.
\end{itemize}
We call the simplices of the dead (resp. living) links {\em dead (resp. living) simplices}.

The main result which we shall require to construct our examples is the following.

\begin{theorem}\label{thm:Main-BNSR}
Assume that for every  vertex $v\in X^{(0)}$ and every (possibly empty) dead simplex $\sigma\in L_{v,X}^{\dagger}$ all ascending and descending living links $Lk_{L_{v,X}^{\ast,\uparrow/ \downarrow}}(\sigma):= L_{v,X}^{\ast,\uparrow /\downarrow}\cap Lk_{L_{v,X}}(\sigma)$ are $(n-\dim(\sigma)-1)$-acyclic, respectively, $L_{v,X}^{\ast,\uparrow /\downarrow}$ is, additionally, $n$-connected.

Then $\chi\in \Sigma^{n+1}(G,\mathbb{Z})$, respectively $\chi\in \Sigma^{n+1}(G)$.
\end{theorem}

Each simplex in $L_{v,X}$ corresponds to a cube $c$ adjacent to $v$. 
For each dead simplex $\sigma$ the corresponding living link $Lk_{L_{v,X}^{\ast,\uparrow/ \downarrow}}(\sigma)$ depends only on the cube corresponding to $\sigma$. 
That is, if a cube $c$ is adjacent to vertices $v, w$ and corresponds to dead simplices $\sigma, \sigma'$ in $L_{v,X}, L_{w,X}$ respectively, then $Lk_{L_{v,X}^{\ast,\uparrow/ \downarrow}}(\sigma) \cong Lk_{L_{w,X}^{\ast,\uparrow/ \downarrow}}(\sigma')$.

Theorem \ref{thm:Main-BNSR} generalises one implication of \cite[Theorem A]{BuxGon-99} by Bux and Gonzalez and the proof follows the same line of argument. We summarise the main steps below for the readers convenience and provide explanations where our situation differs from the one considered in \cite{BuxGon-99}. We closely follow the notation and structure of \cite{BuxGon-99} and refer to this work for further details and illustrations.

Let $X$ be a proper geodesic CAT(0) space and let $Y\subseteq X$ be a closed convex subspace. Let $o_Y^X : X\to Y$ be the closest point (orthogonal) projection onto $Y$; our assumptions guarantee that it is well-defined \cite[Prop. II.2.4]{BriHae-99}.

We define the link of $Y$ in $X$ as the bundle of directions with respect to the orthogonal projection $o_Y^X$. It can be thought of as the unit normal bundle of $Y$ in $X$. To give a precise definition we first define an equivalence relation on geodesics (parametrised by arc length) starting in $Y$, where $\gamma_1\sim \gamma_2$ for $\gamma_i:\left[0,r_i\right]\to X$ if there is an $s>0$ such that $\gamma_1|_{\left[0,s\right]}=\gamma_2|_{\left[0,s\right]}$. We denote the equivalence class of a geodesic $\gamma$ by $\left[\gamma\right]$. The \emph{link of $Y$ in $X$} is the space
\[
	{\rm Lk}_X(Y):= \left\{\left[\gamma\right]\mid \gamma\colon \left[0,r\right]\to X~ \mbox{geodesic with}~ o^{X}_Y(\gamma)=\gamma(0)\in Y\right\}
\]
together with the bundle projection $\pi_Y^X : {\rm Lk}_X(Y)\to Y$, $\left[\gamma\right]\mapsto \gamma(0)$. 

We record the following observation, which is analogous to \cite[Observation 3]{BuxGon-99}.

\begin{remark}\label{rmk:obs3-from-bux-gon}
If $X$ is a CAT(0) cube complex and $Y$ is a convex subcomplex, the cell structure on $X$ induces a cell structure on ${\rm Lk}_X(Y)$, where $[\gamma_1]$ and $[\gamma_2]$ lie in the same cell if there is some $r>0$ such that the images of the restrictions $\gamma_1|_{(0,r)}$ and $\gamma_2|_{(0,r)}$ lie in the interior of the same cell of $X$. If $X$ is equipped with a height function $h:X\to \mathbb{R}$, this cell structure allows us to define the ascending (resp. descending) link ${\rm Lk}_X^{\uparrow}(Y)$ (resp. ${\rm Lk}_X^{\downarrow}(Y)$) as the subcomplex consisting of cells such that for all of their elements $[\gamma]$ there is an $r>0$ such that $(h\circ \gamma)|_{(0,r)}$ is increasing (resp. decreasing). Finally, note that the fibres of $\pi_Y^X$ over points in the same open cube of $Y$ are canonically isomorphic by parallelism of geodesics, equipping the bundles ${\rm Lk}_X(Y)$ and ${\rm Lk}_X^{\uparrow/\downarrow}(Y)$ with a natural product structure over open cubes.
\end{remark}

A key step in the proof of Theorem \ref{thm:Main-BNSR} is that for real numbers $a<b<c$ the preimage $h^{-1}(\left[a,c\right])$ is homotopy equivalent to $h^{-1}(\left[a,b\right])$ with certain ascending links coned off. To make this statement precise we first need to define these links. Given a vertex $v\in \widetilde{X}^{(0)}$ consider the union of all cubes contained in the level set at height $h(v)$ -- it defines a subcomplex of $X$. We denote its connected component containing $v$ by $\widetilde{X}_v^{\dagger}$ and, following \cite{BuxGon-99}, we call the $\widetilde{X}_v^{\dagger}$ \emph{level complexes}.

We will require the following analogue of \cite[Lemma 2]{BuxGon-99}. 
\begin{lemma}\label{lem:convex}
	For every vertex $v\in \widetilde{X}^{(0)}$ the complex $\widetilde{X}_v^{\dagger}$ is a convex subspace of $\widetilde{X}$. In particular, all level complexes of $\widetilde{X}$ are contractible.
\end{lemma}

Since in contrast to \cite{BuxGon-99} we work with CAT(0) cube complexes that are not necessarily universal covers of Salvetti complexes associated to right-angled Artin groups, the proof of Lemma \ref{lem:convex} requires a more general argument than the one used there. To explain the proof, we recall that a subcomplex $M\subseteq N$ of a simplicial complex $N$ is called \emph{full} if whenever a set of vertices of $M$ spans a simplex $\tau\subset N$, then $\tau \subset M$. The main ingredient in the proof of Lemma \ref{lem:convex} is the following result characterising convex subcomplexes of CAT(0) cube complexes.

\begin{lemma}\label{lem:BB-convex}
	Let $X$ be a CAT(0) cube complex and let $Y\subseteq X$ be a cubical subcomplex. Then $Y$ is convex in $X$ if and only if for every vertex $w\in Y$ the link $L_{w,Y}\subseteq L_{w,X}$ is a full subcomplex.
\end{lemma}
\begin{proof}
 This equivalence is proved by Steps 1 to 3 of the proof of \cite[Lemma 8.3]{BesBra-97}.
\end{proof}

\begin{proof}[Proof of Lemma \ref{lem:convex}]
Since the restriction of $h$ to every cube is affine linear, the subcomplex $\widetilde{X}_v^{\dagger}\subseteq \widetilde{X}$ intersects every cube of $\widetilde{X}$ in a (possibly empty) face. This implies that the link $L_{w,\widetilde{X}_{v}^{\dagger}}$ of every vertex $w$ of $\widetilde{X}_{v}^{\dagger}$ is a full subcomplex of the link $L_{w,\widetilde{X}}$ of $\widetilde{X}$ at $w$. By Lemma \ref{lem:BB-convex} this is equivalent to the convexity of $\widetilde{X}_{v}^{\dagger}\subseteq \widetilde{X}$. This completes the proof.
\end{proof}

\begin{lemma}[{Morse Lemma \cite[Lemma 7]{BuxGon-99}}]\label{lem:Morse}
 Let $a<b<c$ be real numbers such that $b-a< C_{min}$. For each level complex $\widetilde{X}_v^{\dagger}$ with $h(\widetilde{X}_v^{\dagger})\in \left[a,b\right)$, let $M_v$ be the mapping cylinder of the bundle projection $Lk_{\widetilde{X}}^{\uparrow}(\widetilde{X}_v^{\dagger})\to \widetilde{X}_v^{\dagger}$, where we view $Lk_{\widetilde{X}}^{\uparrow}(\widetilde{X}_v^{\dagger})$ as a subspace of the level set at height $b$. Let $\widetilde{M}$ be the union of all the $M_v$ obtained in this way. Then $\widetilde{M}\cup h^{-1}(\left[b,c\right])$ is a deformation retract of $h^{-1}(\left[a,c\right])$.
 
 In particular, since all level complexes are contractible, the preimage $h^{-1}(\left[a,c\right])$ is homotopy equivalent to $h^{-1}(\left[b,c\right])$ with the ascending links of level complexes $M_v$ as above coned off.
\end{lemma}

\begin{proof}
The proof is analogous to the one of the Morse Lemma in \cite{BuxGon-99}. By definition of $C_{min}$, we can identify $\widetilde{M}\cup h^{-1}(\left[b,c\right])$ with the subspace of $\widetilde{X}$ obtained by first taking the subcomplex of $\widetilde{X}$ spanned by all vertices of height at least $a$ and then cutting it at height $c$. To prove the assertion we thus need to construct a deformation retraction of the subcomplex of $h^{-1}([a,c])$ consisting of cubes in $\widetilde{X}$ which intersect the $a$-level set in the interior of a face of dimension at least one onto their intersection with $\widetilde{M}\cup h^{-1}(\left[b,c\right])$. Since $\widetilde{X}$ is finite-dimensional such a deformation retraction can be constructed by induction on the dimension of cubes, starting from the highest-dimensional cubes. Indeed, given a convex polyhedron $C$ in $\mathbb{R}^n$, and a codimension one face $F$, there is a deformation retraction of $C$ onto $\partial C\setminus F$ that fixes $\partial C\setminus F$. Here we apply this in the situation where $C$ is the intersection of $h^{-1}([a,c])$ with a cube and $F$ is a codimension one face obtained by intersecting this cube with $h^{-1}(a)$ (see also \cite[Observation 6]{BuxGon-99}).
\end{proof}

As in \cite{BuxGon-99} there is a canonical analogue of Lemma \ref{lem:Morse} for descending links of level complexes. Using completely analogous proofs, we deduce the following analogues of \cite[Corollary 8]{BuxGon-99} and \cite[Theorem 9]{BuxGon-99}.

\begin{corollary}
	If all ascending and descending links of level complexes of $\widetilde{X}$ are $n$-connected (resp. $n$-acyclic), then for every closed interval $J$ the preimage $h^{-1}(J)\subseteq \widetilde{X}$ is $n$-connected (resp. $n$-acyclic).
\end{corollary}

\begin{theorem}\label{thm:Prelim-BNSR}
 Let $h:\widetilde{X}\to \mathbb{R}$ be a height map and $\chi : G=\pi_1(\widetilde{X})\to \mathbb{R}$ an associated character. Assume that for every vertex $v\in \widetilde{X}^{(0)}$ the ascending links $Lk^{\uparrow}_{\widetilde{X}}(\widetilde{X}^{\dagger}_v)$are $n$-acyclic (resp. $n$-connected). Then $\chi \in \Sigma^{n+1}(G,\mathbb{Z})$, resp. $\chi\in \Sigma^{n+1}(G)$.
 
 Analogously, if for every vertex $v\in \widetilde{X}^{(0)}$ the descending links $Lk^{\downarrow}_{\widetilde{X}}(\widetilde{X}^{\dagger}_v)$ are $n$-acyclic (resp. $n$-connected), then $-\chi \in \Sigma^{n+1}(G,\mathbb{Z})$, resp. $-\chi\in \Sigma^{n+1}(G)$.
\end{theorem}

We now deduce Theorem \ref{thm:Main-BNSR} from Theorem \ref{thm:Prelim-BNSR} and the following results, which generalise one of the implications of Lemmas 12 and 13 in \cite{BuxGon-99}. 

\begin{lemma}\label{lem:BG-12}
	Let $v\in \widetilde{X}^{(0)}$ be a vertex and let $\widetilde{X}_v^{\dagger}$ be the associated level complex. 
	If for every cube in $\widetilde{X}_v^{\dagger}$ its preimage in $Lk^{\uparrow}_{\widetilde{X}}(\widetilde{X}^{\dagger}_v)$ (resp. $Lk^{\downarrow}_{\widetilde{X}}(\widetilde{X}^{\dagger}_v)$) under the bundle projection is $n$-acyclic/ $n$-connected, then $Lk^{\uparrow}_{\widetilde{X}}(\widetilde{X}^{\dagger}_v)$ (resp. $Lk^{\downarrow}_{\widetilde{X}}(\widetilde{X}^{\dagger}_v)$) is $n$-acyclic/ $n$-connected.
\end{lemma}

\begin{lemma}\label{lem:BG-13}
	Assume that for every vertex $v\in X^{(0)}$ and every (possibly empty) dead simplex $\sigma\in L_{v,X}^{\dagger}$ all ascending living links $Lk_{L_{v,X}^{\ast,\uparrow}}(\sigma)$ are $(n-\dim(\sigma)-1)$-acyclic, respectively, $L_{v,X}^{\ast,\uparrow}$ is, additionally, $n$-connected. Then for every cube in $\widetilde{X}_v^{\dagger}$ its preimage in $Lk^{\uparrow}_{\widetilde{X}}(\widetilde{X}^{\dagger}_v)$ under the bundle projection is $n$-acyclic, respectively $n$-connected.
	
	An analogous result holds for descending (living) links.
\end{lemma}

The main idea in the proof of Lemmas \ref{lem:BG-12} and \ref{lem:BG-13} is that for every vertex $v\in \widetilde{X}^{(0)}$ we can reconstruct the ascending link $Lk^{\uparrow}_{\widetilde{X}}(\widetilde{X}^{\dagger}_v)$ from the preimages of midpoints of cubes in $\widetilde{X}_v^{\dagger}$ under the bundle projection $\pi_{\widetilde{X}_v^{\dagger}}^{\widetilde{X},\uparrow}:Lk^{\uparrow}_{\widetilde{X}}(\widetilde{X}^{\dagger}_v)\to \widetilde{X}^{\dagger}_v$. Since the preimages of such midpoints can be canonically identified with ascending living links $Lk_{L_{w,\widetilde{X}}^{\ast,\uparrow}}(\sigma)$ of dead simplices, this allows us to determine the connectivity properties of $Lk^{\uparrow}_{\widetilde{X}}(\widetilde{X}^{\dagger}_v)$ using Mayer--Vietoris and Seifert van Kampen type arguments.

\begin{proof}[Proof of Lemma \ref{lem:BG-12}]
The proof is completely analogous to the proof of the only if direction in Lemma 12 of \cite{BuxGon-99}.
\end{proof}

\begin{proof}[Proof of Lemma \ref{lem:BG-13}]

 The proof follows along the same lines as the proof that (2) implies (1) in Lemma 13 of \cite{BuxGon-99}. However, a bit more care is required, as the links of different vertices in $\widetilde{X}$ may be different. We thus provide an adapted version of the argument that takes this into consideration.
 
 We first observe that the connectivity properties of the complexes $Lk_{L_{v,X}^{\ast,\uparrow}}(\sigma)$ induce analogous connectivity properties for the complexes $Lk_{L_{v,\widetilde{X}}^{\ast,\uparrow}}(\sigma)$, which are defined in the same way for all vertices $v\in \widetilde{X}^{(0)}$.
 
 The proof consists of reconstructing the preimage of a (closed) cube of $X_{v}^{\dagger}$ under the bundle projection $\pi_{\widetilde{X}^{\dagger}_v}^{\widetilde{X},\uparrow}:{\rm Lk}^{\uparrow}_{\widetilde{X}}(\widetilde{X}^{\dagger}_v)\to \widetilde{X}^{\dagger}_v$ from the ascending living links $Lk_{L_{v,\widetilde{X}}^{\ast,\uparrow}}(\sigma)$ of (possibly empty) dead simplices $\sigma\in L_{v,\widetilde{X}}^{\dagger}$ for $v\in \widetilde{X}^{(0)}$ by gluings along subcomplexes.
 
 For this we first observe that the preimage $(\pi_{\widetilde{X}^{\dagger}_v}^{\widetilde{X},\uparrow})^{-1}(x)$ of an interior point $x$ of a cube $c$ in $\widetilde{X}_{v}^{\dagger}$ is naturally identified with the ascending living link ${\rm Lk}_{L_{w,\widetilde{X}}}^{\ast,\uparrow}(\sigma)$, where $w$ is a vertex of $c$ and the simplex $\sigma\in L_{w,\widetilde{X}}^{\dagger}$ corresponds to the cube $c$. In particular, $(\pi_{\widetilde{X}^{\dagger}_v}^{\widetilde{X},\uparrow})^{-1}(x)$ is $(n-\dim(\sigma)-1)$-acyclic by our assumptions.
 
 We will now denote by $m_c$ the midpoint of a cube $c$. Let $c=c_1\times c_2\times\left[0,1\right]$ be a decomposition of a cube $c$ of $\widetilde{X}_v^{\dagger}$ as a direct product, where we identify $\left[0,1\right]$ with one of the edges of $c$. Then 
 {\small
 \begin{align*}
 	&\left(\pi_{\widetilde{X}^{\dagger}_v}^{\widetilde{X},\uparrow}\right)^{-1}\left(c_1\times m_{c_2}\times \left[0,1\right]\right)=\\
 	&\left(\pi_{\widetilde{X}^{\dagger}_v}^{\widetilde{X},\uparrow}\right)^{-1}\left(c_1\times m_{c_2}\times \left\{0\right\}\right) \times\left\{0\right\}
 	\bigcup \left(\pi_{\widetilde{X}^{\dagger}_v}^{\widetilde{X},\uparrow}\right)^{-1}\left(c_1\times m _{c_2}\times \left\{\frac{1}{2}\right\}\right)\times \left[0,1\right]
 	\bigcup \left(\pi_{\widetilde{X}^{\dagger}_v}^{\widetilde{X},\uparrow}\right)^{-1}\left(c_1\times m_{c_2}\times \left\{1\right\}\right) \times\left\{1\right\}\\
 	&\simeq \left(\pi_{\widetilde{X}^{\dagger}_v}^{\widetilde{X},\uparrow}\right)^{-1}\left(c_1\times m_{c_2}\times \left\{0\right\}\right) \bigcup_{\left(\pi_{\widetilde{X}^{\dagger}_v}^{\widetilde{X},\uparrow}\right)^{-1}\left(c_1\times m _{c_2}\times \left\{\frac{1}{2}\right\}\right)} \left(\pi_{\widetilde{X}^{\dagger}_v}^{\widetilde{X},\uparrow}\right)^{-1}\left(c_1\times m_{c_2}\times \left\{1\right\}\right),
 \end{align*}}
 where we glue along cells, using that there are natural embeddings $$\left(\pi_{\widetilde{X}^{\dagger}_v}^{\widetilde{X},\uparrow}\right)^{-1}\left(c_1\times m_ {c_2}\times \left\{\frac{1}{2}\right\}\right)\hookrightarrow \left(\pi_{\widetilde{X}^{\dagger}_v}^{\widetilde{X},\uparrow}\right)^{-1}\left(c_1\times m_{c_2}\times \left\{0\right\}\right)$$ 
 and 
 $$\left(\pi_{\widetilde{X}^{\dagger}_v}^{\widetilde{X},\uparrow}\right)^{-1}\left(c_1\times m_{c_2}\times \left\{\frac{1}{2}\right\}\right)\hookrightarrow \left(\pi_{\widetilde{X}^{\dagger}_v}^{\widetilde{X},\uparrow}\right)^{-1}\left(c_1\times m_{c_2}\times \left\{1\right\}\right)$$ of cell complexes; here $\simeq$ denotes a homotopy equivalence. Note that analogous decompositions hold when $c_1=\emptyset$ or $c_2=\emptyset$ (i.e. $c= c_1\times[0,1]$ or $c=c_2\times [0,1]$).
 
 The above decompositions allow us to inductively reconstruct the preimage of a cube $c$ of $\widetilde{X}_v^{\dagger}$ under $\pi_{\widetilde{X}^{\dagger}_v}^{\widetilde{X},\uparrow}$ from the preimages of midpoints of cubes under $\pi_{\widetilde{X}^{\dagger}_v}^{\widetilde{X},\uparrow}$ by taking direct product decompositions of faces of $c$ (the induction is over the dimension of $c_1$). From the Meyer-Vietoris sequence, we see that if $X_1$ and $X_2$ are $m$-acyclic spaces and $Y\subseteq X_1,~X_2$ is a common subspace which is $(m-1)$-acyclic, then $X_1\cup_Y X_2$ is $m$-acyclic. Together with the assumptions on the connectivity properties of ascending living links, we see that $\left(\pi_{\widetilde{X}^{\dagger}_v}^{\widetilde{X},\uparrow}\right)^{-1}(c)$ is $n$-acyclic.
 
 The statement on $n$-connectivity follows analogously by using the Seifert van Kampen Theorem to show $1$-connectivity and then applying Hurewicz's Theorem to deduce $n$-connectivity.

\end{proof}

\section{Coabelian subgroups from ramified covers of cube complexes}\label{sec:Brady-type-examples}
In this Chapter we will explain how a special case of Kropholler's generalisation in \cite{Kro-21} of Brady's construction in \cite{Bra-99} allows us to produce examples of subgroups of hyperbolic groups of type $F_2$ and not $F_3$ of arbitrarily large corank. This provides the first proof of Theorem \ref{thm:Main}.

\subsection{Hyperbolic covers of 3-dimensional cube complexes}\label{sec:Brady-examples}

Our first construction of the subgroups of hyperbolic groups in Theorem \ref{thm:Main} arises as special case of Kropholler's generalisation \cite{Kro-21} of Brady's construction \cite{Bra-99}. In this section we will present these groups and summarise their most important properties. The proofs of all results presented here can be found in \cite{Bra-99,Kro-21}. It will be clear from our presentation that it can be varied to produce more examples with the same properties by varying the input data.

For $n\geq 2$ denote by $\Theta^{(n)}$ the graph with two vertices $\left\{0,1\right\}$ and $2n$ edges obtained by taking $2n$ copies of the unit interval $\left[0,1\right]$ and identifying all endpoints labelled 0 and all endpoints labelled 1. We label the edges $x_1,~y_1,~\dots, x_n,~y_n$ and fix an auxiliary orientation such that all edges labelled $x_i$ are oriented from $0$ to $1$ and all edges labelled $y_i$ from $1$ to $0$. 

Let $X^{(n)}= \Theta^{(n)}_1\times \Theta^{(n)}_2\times \Theta^{(n)}_3$ be the direct product of three copies of $\Theta^{(n)}$ and let $K^{(n)}=\left(0 \times 1 \times \Theta^{(n)}_3\right) \cup \left(\Theta^{(n)}_1\times 0 \times 1\right) \cup \left(1\times \Theta^{(n)}_2\times 0\right)$. 

By \cite[Section 3.1]{Kro-21}, for any prime $p>2n$ there is a $p^3$-fold ramified cover $f: Y^{(n)}\to X^{(n)}$ with ramification locus $K^{(n)}$ such that $Y^{(n)}$ admits a natural cubical structure for which $f$ is a combinatorial map of cube complexes and such that the induced metric on $Y^{(n)}$ is hyperbolic. More precisely, in a tubular neighbourhood of every connected component of $K^{(n)}$ the map $f$ is a $p$-fold ramified covering map and it is regular everywhere else. Since all links of $X$ are joins $F_1\ast F_2\ast F_3$ for $F_i=F=L_{v,\Theta^{(n)}}$ a discrete set with $2n$ elements labelled by the edges of $\Theta^{(n)}$ adjacent to $v$, this yields the following description of the links of vertices in $Y^{(n)}$:
 
 \begin{lemma}[{\cite{Bra-99,Kro-21}}]\label{lem:Links}
  Let $v\in Y^{(n)}$ be a vertex of $Y^{(n)}$. Then $L_{v,Y^{(n)}}$ admits the following description:
  \begin{enumerate}
   \item if $f(v)\notin K^{(n)}$, then $L_{v,Y^{(n)}}\cong L_{f(v),X^{(n)}}\cong F_1\ast F_2\ast F_3$;
   \item if $f(v)\in K^{(n)}$, then $L_{v,Y^{(n)}}$ is a $p$-fold connected ramified cover of $L_{f(v),X^{(n)}}$, which (up to reordering factors) is given by $\Gamma \ast F_{3}\to F_{1}\ast F_{2}\ast F_{3}$, where the map $\Gamma\to F_{1}\ast F_{2}$ on the first two factors is a connected regular $p$-fold cover for which the preimage of every 4-cycle in $F_1\ast F_2$ is a $4p$-cycle in $\Gamma$ and $F_{3}$ is the ramification locus.
  \end{enumerate}
 \end{lemma}
 
 We call a vertex satisfying (1) of \emph{Type 1} and a vertex satisfying (2) of \emph{Type 2}.
 
\subsection{Homomorphisms to free abelian groups} 
 
We now fix $n\geq 2$ and $p>2n$ a prime. We retain the notation from Section \ref{sec:Brady-examples}, except that we will usually drop the upper index $~^{(n)}$ whenever this leads to no confusion. For $1\leq j \leq 3$ we define $n-1$ continuous maps $h_{i,j}:\Theta^{(n)}_j\to S^1=\left[0,1\right]/0\sim 1$, $2\leq i \leq n$, by

\[
	t\mapsto \left\{
	\begin{array}{ll} 
		t& ,~t\in x_1\cup x_i\\
		1-t& ,~t\in y_1\cup y_i\\
		0 & ,~\mbox{else}
	\end{array} 
	\right.
\]
Then the induced homomorphisms $\phi_{i,j}:=h_{i,j,\ast}: \pi_1(\Theta^{(n)}_j)\to \pi_1(S^1)\cong \mathbb{Z}$ have image of index 2 if $n=2$ and are surjective if $n>2$. Moreover, the subset $\left\{\phi_{2,j},~\dots,\phi_{n,j}\right\}\subset {\rm Hom}(\pi_1(\Theta^{(n)}_j),\mathbb{R})$ is linearly independent. 

By extending affine linearly to cubes, we obtain continuous maps $h_i:=\sum_{j=1}^3 h_{i,j}:X\to S^1$ and linearly independent non-trivial homomorphisms $\phi_i:=\sum_{j=1}^3 \phi_{i,j}: \pi_1(X)\to \mathbb{Z}$. In particular, the continuous map $h:= \left(h_2,\dots, h_{n}\right): X\to (S^1)^{n-1}$ induces a homomorphism $\phi=h_{\ast}=\left(\phi_2,\dots,\phi_{n}\right): \pi_1(X) \to \mathbb{Z}^{n-1}$ with free abelian image of rank $n-1$. By composing with the ramified covering $f: Y\to X$ we define $\overline{h}_i:=h_i \circ f$, $\overline{h}:=h\circ f=\left(\overline{h}_2,\dots,\overline{h}_{n}\right)$, $\overline{\phi}_i:= \overline{h}_{i,\ast}: \pi_1(Y)\to \pi_1(S^1)$, $\overline{\phi}:=\overline{h}_{\ast}=\left(\overline{\phi}_2,\dots,\overline{\phi}_{n}\right)$. Theorem \ref{thm:Main} is a direct consequence of the following theorem:
\begin{theorem}
	The image of the homomorphism $\overline{\phi}:\pi_1(Y)\to \mathbb{Z}^{n-1}$ has finite index and $\ker(\overline{\phi})\leq \pi_1(Y)$ is a finitely presented subgroup of the hyperbolic group $\pi_1(Y)$ which is not of type $\mathcal{F}_3$.
\end{theorem}

\begin{proof}
	For the first assertion, note that $f_*\colon \pi_1(Y)\to \pi_1(X)$ has finite index image, since $f$ is a ramified cover. Thus, $\phi\circ f_* = \overline{\phi}(\pi_1(Y))$ has finite index image in $\phi(\pi_1(X))$ and hence finite index in $\Z^{n-1}$.
	
	For the second assertion, by Theorem \ref{thm:Sigma-BNSR}, it suffices to show that $S(\pi_1(Y),\ker(\overline{\phi}))\subseteq \Sigma^2(\pi_1(Y))$. We will use Theorem \ref{thm:Main-BNSR} to prove this. First observe that every element of $S(\pi_1(Y),\ker(\overline{\phi}))$ has a representative of the form $\sum_{i=2}^{n}\lambda_i \overline{\phi}_i:\pi_1(Y)\to \mathbb{R}$ with $\left(\lambda_2,\dots,\lambda_{n}\right)\in \mathbb{R}^{n-1}\setminus \left\{0\right\}$ and, conversely, any such homomorphism defines an element of $S(\pi_1(Y),\ker(\overline{\phi}))$. 
	
	Let $\left[\overline{\psi}\right]=\left[\sum_{i=2}^{n}\lambda_i \overline{\phi}_i\right]\in S(\pi_1(Y),\ker(\overline{\phi}))$. We may assume that $\lambda_2\neq 0$; the other cases being analogous. To define a $\psi$-equivariant height map $\widetilde{\overline{h}}: \widetilde{Y}\to \mathbb{R}$ we proceed as follows. We first equip the edges of $\Theta$ with the following (signed) weights:
	\begin{itemize}
		\item the edges labelled $x_1$ and $y_1$ have weights $\sum_{i=2}^{n}\lambda_i$; and
		\item the edges labelled $x_i$ and $y_i$ have weights $\lambda_i$.
	\end{itemize}
We then equip $X$ with weights via the canonical identifications of $\Theta^{(n)}_i$ with $\Theta$. Finally, we equip $Y$ and $\widetilde{Y}$ with weights by lifting the weights assigned to $X$ with respect to the canonical composition $\widetilde{Y}\to Y\to X$. By definition these weights then equip $\widetilde{Y}$ with a $\psi$-equivariant height map $\widetilde{\overline{h}}:\widetilde{Y}\to\mathbb{R}$, where edges are mapped to $\mathbb{R}$ according to their weights (the sign of a weight determines if an edge is mapped according to its orientation or against its orientation).\footnote{This map is uniquely determined up to a translation of $\mathbb{R}$.} 

Note that $\overline{\psi}$ descends to a homomorphism $\psi=\sum_{i=2}^n\lambda_i\phi_i: \pi_1(X)\to \mathbb{R}$ and $\widetilde{\overline{h}}$ descends to a $\psi$-equivariant height map $\widetilde{h}: \widetilde{X}\to \mathbb{R}$

To apply Theorem \ref{thm:Main-BNSR}, for every vertex $v$ of $\widetilde{Y}$ we need to determine the connectivity properties of all ascending and descending living links of all dead simplices in $L_{v,\widetilde{Y}}$. We start by observing that the ascending and descending links of v in $\widetilde{Y}$ are the intersection of $L_{v,\widetilde{Y}}$ with the preimages of the ascending and descending living links of vertices in $X$ under the composition $\widetilde{Y}\to Y\to X$. Moreover, we can determine the ascending and descending links from the induced weights and orientations on $Y$ and $X$. 

We first analyse the ascending and descending links in $X$. For every vertex $v\in X$ we have $L^{\uparrow/\downarrow}_{v,X}=A_1^{+/-}\ast A_2^{+/-}\ast A_3^{+/-}$ with $A_i^{+/-}=L^{\uparrow/\downarrow}_{pr_i(v),\Theta^{(n)}_i}$, where $pr_i: X\to \Theta^{(n)}_i$ is the projection onto $\Theta^{(n)}_i$. 

To simplify notation for the remainder of the proof, we will now only analyse the ascending links. The arguments for the descending links are completely analogous.

By definition of $\widetilde{h}:\widetilde{X}\to\mathbb{R}$, at every vertex $v\in X$ there are at least two ascending edges. Thus, $|A_i^{+}|\geq 2$ and it is immediate that the ascending and descending links of the empty dead simplex are homotopy equivalent to wedges of 2-spheres and thus 1-connected.

Let now $\left\{d_1\right\}$ be a dead vertex of $ L_{v,X}$. We may assume that $d_1 \in F_1\setminus (A_1^{+}\cup A_1^{-})$ with the other cases being analogous. Then 
\[
{\rm Lk}_{L_{v,X}^{\ast,\uparrow}}(\left\{d_1\right\})=L_{v,X}^{\ast,\uparrow}\cap {\rm Lk}_{L_{v,X}}(\left\{d_1\right\})= (A_1^{+}\ast A_2^{+} \ast A_3^{+}) \cap (F_2\ast F_3)= A_2^{+}\ast A_3^{+}
\]
is a wedge of circles and thus $0$-connected.

Finally, if $\left\{d_1,d_2\right\}$ is a dead edge of $ L_{v,X}$ we may assume that $d_i\in F_i \setminus (A_i^{+}\cup A_i^{-})$; the other cases are again analogous. Then ${\rm Lk}_{L_{v,X}^{\ast,\uparrow}}(\left\{d_1,d_2\right\})=A_3^{+}\neq \emptyset$ is $(-1)$-connected. Since $v\in X$ was an arbitrary vertex, this shows that $\psi\in \Sigma^2(\pi_1(X))$.

We now turn our attention back to $Y$. Let $v\in Y$ be a vertex. 

If $v$ is of Type 1, then all of its associated links are isomorphic to those of $f(v)\in X$. Thus, for Type 1 vertices it follows from the above discussion that ${\rm Lk}_{L_{v,Y}^{\ast,\uparrow}}(\sigma)$ is $(1-\dim(\sigma)-1)$-connected for every dead simplex $\sigma$ of $L_{v,Y}$. 

Assume now that $v\in Y$ is a vertex of Type 2. The ascending living link of the empty simplex is the preimage of $L_{f(v),X}^{\ast,\uparrow}=A_1^{+}\ast A_2^{+}\ast A_3^{+}$ under the map on links induced by $f$, which (up to a permutation of the factors) is given by
\[
{\rm Lk}_{L_{v,Y}^{\ast,\uparrow}}(\emptyset)=L_{v,Y}^{\ast,\uparrow}=(\Gamma\ast F_3)^{\uparrow}=\Gamma^{\uparrow}\ast F_3^{\uparrow}=\Gamma^{\uparrow}\ast A_3^{+}.
\] 
 By Lemma \ref{lem:Links} the preimage of every $4$-cycle of $A_1^{+}\ast A_2^{+}$ in $\Gamma$ is a (connected) $4p$-cycle. As in \cite[Section 3.1]{Kro-21} we can cover the graph $A_1^{+}\ast A_2^{+}$ with a sequence of 4-cycles such that every 4-cycle has non-empty intersection with the union of all previous 4-cycles. It follows that the sequence of their preimages in $\Gamma^{\uparrow}$ consists of $4p$-cycles with the same property. In particular, their union is 0-connected and coincides with $\Gamma^{\uparrow}$. Thus, ${\rm Lk}_{L_{v,Y}^{\ast,\uparrow}}=\Gamma^{\uparrow}\ast A_3^{+}$ is homotopy equivalent to a (1-connected) wedge of 2-spheres.

For a dead 0-simplex $\left\{d\right\}$ of $L_{v,Y}^{\dagger}$ we have to consider two cases. Case 1 is that $f(d)\in F_1^{\dagger}\ast F_2^{\dagger}$ and Case 2 is that $f(d)\in F_3^{\dagger}$. In Case 2 the set $f^{-1}(f(d))=\left\{d\right\}$ consists of the single point $d$. It is then easy to see that ${\rm Lk}_{L_{v,Y}^{\ast,\uparrow}}(\left\{d\right\})=\Gamma^{\uparrow}$, which is 0-connected by the above discussion.  In Case 2 we may assume that $f(d)\in F_1^{\dagger}$. Since $L_{v,Y}=\Gamma \ast F_3$, there is an edge from $d$ to every point in $A_3^{+}$. Using that $\Gamma\to F_1\ast F_2$ is a covering map we see that there is at least one vertex $b\in \Gamma^{\uparrow}\cap {\rm Lk}_{L_{v,Y}^{\ast,\uparrow}}(\left\{d\right\})$. Necessarily for every such vertex $b$ we have $f(b)\in A_2^{+}$. We deduce that there is a non-empty subset $B_2\subset \Gamma$ with $f(B_2)\subset A_2^+$ such that ${\rm Lk}_{L_{v,Y}^{\ast,\uparrow}}(\left\{d\right\})=B_2\ast A_3^{+}$. In particular, ${\rm Lk}_{L_{v,Y}^{\ast,\uparrow}}(\left\{d\right\})$ is $0$-connected.

Finally, one readily checks that for every dead edge $\left\{d_1,d_2\right\}$ of $L_{v,Y}^{\dagger}$ we have ${\rm Lk}_{L_{v,Y}}^{\ast,\uparrow}(\left\{d_1,d_2\right\})\neq \emptyset$, because the same is true for ${\rm Lk}_{L_{f(v),X}}^{\ast,\uparrow}(\left\{f(d_1),f(d_2)\right\})$ and $L_{v,Y}\to L_{f(v),X}$ is a connected ramified cover.

This implies that ${\rm Lk}_{L_v^{\ast,\uparrow/\downarrow}}(\sigma)$ is $(1-\dim(\sigma)-1)$-connected for every dead simplex $\sigma$ of $L_v$. Theorem \ref{thm:Main-BNSR} then implies that $[\overline{\psi}]\in \Sigma^2(\pi_1(Y))$. Since $[\overline{\psi}]\in S(\pi_1(Y),\ker(\overline{\phi}))$ was arbitrary, we deduce that $S(\pi_1(Y),\ker(\overline{\phi}))\subseteq \Sigma^2(\pi_1(Y))$. In particular, by Theorem \ref{thm:Sigma-BNSR} $\ker(\overline{\phi})$ is finitely presented.

We still need to explain why $\ker(\overline{\phi})$ is not of type $F_3$. Suppose that the kernel was of type $F_3$, then it would be of type $FP$ and hence $Y$ would have Euler characteristic 0 \cite[Proposition IX.7.3 (d)]{Bro-82}. However, by \cite[Proposition 3.3]{Kro-21} the Euler characteristic of $Y$ is $p^3(2-18n+24n^2-8n^3)-3p^2(2n-2)$. Since $n\geq2$ we have $2-18n+24n^2-8n^3 < 0$ and so the Euler characteristic is negative.
\end{proof}

\begin{remark}
 Observe that our proof only used that the ascending and descending links of the restriction of every height map associated with a homomorphism $\psi$ to $\Theta$ both have at least two elements. Thus, our proof can be adapted to produce large families of examples of normal subgroups of hyperbolic groups of type $F_2$ and not $F_3$ with quotient $\mathbb{Z}^m$, for every $m\geq 1$, by applying it to the families of examples constructed in \cite{Kro-21}.
\end{remark}

\section{Coabelian subgroups of Lodha's examples and homological finiteness properties}\label{sec:Lodha-type-examples}

In \cite{Lod-18}, Lodha constructed another example of a finitely presented non-hyperbolic subgroup of a hyperbolic group. The example is constructed by first defining an explicit hyperbolic CAT(0) cube complex  contained in a direct product of three copies of a certain graph and then putting a Morse function on it. 

Here we present a variation of this construction, which produces examples with spheres of arbitrarily large dimension in $\Sigma^2(G)$ and thus finitely presented subgroups of hyperbolic groups arising as kernels of morphisms onto free abelian groups of arbitrarily large rank. To define our examples, we will start from a bipartite graph $\Gamma$ with certain additional properties. We use this graph $\Gamma$ to define a hyperbolic CAT(0) cube complex $X_{\Gamma}$ contained in a direct product of three complete bipartite graphs. Finally, we produce a family of height maps on $X_{\Gamma}$ which induce linearly independent characters on $\pi_1(X_{\Gamma})$ that span a high-dimensional sphere in $\Sigma^2(G)$.

\subsection{Definition of sizeable graphs of rank $n$}

We start by introducing the family of input graphs $\Gamma$ for our construction.

\begin{definition}
	Let $\G$ be a simplicial, bipartite graph on $A\sqcup B$. We say that $\G$ is {\em $n$-Morse suited} if for each $i\in \{1, \dots, n\}$ and $t\in\{-, +\}$, there are subsets $A_i^t\subset A, B_i^t\subset B$ such that: 
	\begin{itemize}
		\item $A_i^t\cap A_j^s \neq \emptyset$ if and only if $i = j$ and $s = t$.
		\item $B_i^t\cap B_j^s \neq \emptyset$ if and only if $i = j$ and $s = t$.
		\item $A = \cup_{i, s}A_i^s, B = \cup_{j, t}B_j^t$. 
	\end{itemize}
\end{definition}

\begin{definition}
	We say that an $n$-Morse suited graph $\G$ is a {\em sizeable graph of rank $n\in \mathbb{N}$}, if $\G$ contains no embedded cycles of length 4 and for each choice of $i, j\in \{1, \dots, n\}$ and $s, t\in \{-, +\}$ we have that the subgraph spanned by $A_i^s\sqcup B_j^t$ is connected. 
\end{definition}

One should compare this to the definition of a sizeable graph from \cite{Kro-21}, which correspond to sizeable graphs of rank 1 in our definition. The conditions serve various purposes. The bipartite condition will be used in the construction of the cube complex $X_{\Gamma}$ and to define a function to $\Z^n$. The condition on cycles of length 4 will be used to ensure hyperbolicity of $\pi_1(X_{\Gamma})$. The other conditions will ensure that certain ascending and descending links have appropriate connectivity properties. 

\subsection{Existence of sizeable graphs of rank $n$}

Before continuing we give a method to construct sizeable graphs of rank $n$. This method uses a similar technique to that employed in \cite[Definition 8]{Lod-18}. This construction is not optimal to create small cube complexes. One could employ the techniques used in the appendix of \cite{Kro-21} to give minimal examples of sizeable graphs of rank $n$.

\begin{definition}
	We say that an $n$-Morse suited graph $\G$ is {\em modular} if there are bijections $a_i^s\colon \Z/p\Z \to A_i^s$ and $b_j^t\colon \Z/p\Z\to B_j^t$, and for each choice of $i, j, s, t$ there is a set of integers $\tau$ such that there is an edge from $a_i^s(k)$ to $b_j^t(l)$ if and only if $l-k$ is congruent to an element of $\tau$ modulo $p$. 
\end{definition}

The advantage of modular $n$-Morse suited graphs is that checking that they are sizeable comes down to checking certain equalities modulo $p$. 
Moreover, defining a modular $n$-Morse suited graph is equivalent to assigning to each edge of the complete bipartite graph on $\{A_1^-, A_1^+,\dots, A_n^-, A_n^+\}\sqcup\{B_1^-, B_1^+,\dots, B_n^-, B_n^+\}$ a subset of the integers. 

\begin{construction}
	Let $\mathcal{K}$ be the complete bipartite graph on $$\{A_1^-, A_1^+,\dots, A_n^-, A_n^+\}\sqcup\{B_1^-, B_1^+,\dots, B_n^-, B_n^+\}.$$
	Let $E$ be the set of edges of $\mathcal{K}$ and pick an ordering $E = \{e_1, \dots, e_{(2n)^2}\}$ on them.
	We define a function $\sigma\colon E\to \mathcal{P}(\Z)$ as follows:
	\begin{itemize}
		\item $\sigma(e_1) = \{0, 1\}$. 
		\item For each $i$, let $m_i$ be the sum of all the elements of the $\sigma(e_j)$ with $j<i$. 
		\item Define $\sigma(e_i) = \{2m_i, 4m_i\}$. 
	\end{itemize}
	Let $p$ be any prime larger than $8m_{(2n)^2}$. 
	Let $\G$ be the modular $n$-Morse suited graph given by this data. 
	
	We will now prove that $\G$ is a sizeable graph of rank $n$. 
	We begin by showing that the subgraph spanned by $A_i^s\sqcup B_j^t$ is connected for all choices of $i, j, s, t$. 
	Fix $i, j, s, t$ and denote the subgraph spanned by $A_i^s\sqcup B_j^t$ by $\mathcal{S}$.
	Since the sets assigned to edges by the map $\sigma$ all have size 2, we see that every vertex in $\mathcal{S}$ has valence 2. 
	Thus $\mathcal{S}$ is a disjoint union of circles. 
	There is an action of $\Z/p\Z$ on $\G$ given by $m\cdot a_i^s(k) = a_i^s(k+m)$ and $ m\cdot b_j^t(l) = b_j^t(l+m)$. 
	This action preserves $\mathcal{S}$ and hence preserves its components. 
	Since this is a transitive action of $\Z/p\Z$ on $A_i^s$ we see that each component of $\mathcal{S}$ either contains one vertex of $A_i^s$ or contains all of $A_i^s$. 
	Since each vertex has valence two we see that any component of $\mathcal{S}$ contains at least two vertices and hence contains all of $A_i^s$. 
	Similarly, it must contain all of $B_j^t$ and is a cycle of length $2p$, hence connected. 
	
	We now move on to showing that there are no embedded cycles of length $4$. 
	Suppose for a contradiction that there is an embedded cycle of length 4. 
	Label the vertices of this cycle $a_i^s(k), b_j^t(l), a_{i'}^{s'}(k'), b_{j'}^{t'}(l')$. 
	We split into the following four cases: 
	\begin{enumerate}
		\item $i = i', s = s', j = j', t = t'$, 
		\item $i = i', s = s'$, 
		\item $j = j', t = t'$, 
		\item $A_i^s\neq A_{i'}^{s'}, B_j^t\neq B_{j'}^{t'}$. 
	\end{enumerate}
	
	The first case would give a cycle of length 4 in the subgraph spanned by $A_i^s, B_j^t$. 
	But by the above this subgraph is a loop of length $2p>4$, leading to a contradiction. 
	
	The second case corresponds to two edges in the graph $\mathcal{K}$. These are the edges between  $A_i^s, B_j^t$ and $A_i^s, B_{j'}^{t'}$; let these edges be $e_\alpha$, $e_\beta$ respectively. 
	Without loss of generality we can assume that $\beta > \alpha$. 
	By considering the possible orderings on the edges, the existence of a 4-cycle with these vertices would require $2m_\beta +2m_\alpha \equiv 0 \mod p$ or $2m_\beta - 2m_\alpha \equiv 0 \mod p$.
	However, by the definition of $m_\beta$ and $p$ we see that neither of these can hold. 
	Indeed,  $0<2m_\beta +2m_\alpha<4m_\beta < p$ and $0<2m_\beta - 2m_\alpha < 2m_\beta < p$. 
	The third case can be treated identically. 
	
	For the fourth case, we consider the 4 edges $e_\alpha, e_\beta, e_\gamma, e_\delta$ of $\mathcal{K}$ that are defined by the vertices $a_i^s(k), b_j^s(l), a_{i'}^{s'}(k'), b_{j'}^{t'}(l')$. 
	We can assume that $\alpha >\beta, \gamma, \delta$. 
	Let $r_\alpha, r_\beta, r_\gamma, r_\delta$ be the elements of $\sigma(e_\alpha),\sigma(e_\beta), \sigma(e_\gamma), \sigma(e_\delta)$ used in this loop.
	To have an embedded cycle of length 4 we would then have to have an equality of the form $r_\alpha \pm r_\beta\pm r_\gamma\pm r_\delta \equiv 0 \mod p$ (after possibly reversing orientations if needed). However, we have inequalities $0< r_\alpha - r_\beta- r_\gamma- r_\delta <r_\alpha + r_\beta+ r_\gamma+ r_\delta < 2r_\alpha < p$, which is a contradiction.
	Hence there are no embedded cycles of length 4. 
\end{construction}

\subsection{Proof of Theorem \ref{thm:Main} using cube complexes defined by sizeable graphs}

We now fix a sizeable graph $\Gamma$ of rank $n$ and define the corresponding cube complex $X_\G$. Let $A\ast B$ be the complete bipartite graph on $A$ and $B$. Consider the cube complex $\mathbb{X} = (A*B)\times (A*B)\times (A*B)$. There is a partition of the vertex set of $\mathbb{X}$ into 8 types, corresponding to whether each coordinate is in $A$ or $B$. By definition $X_{\G}$ will be the subcomplex of $\mathbb{X}$ induced by its vertex set. It thus suffices to define the vertex set of $X_{\G}$. A vertex $v = (v_1, v_2, v_3)$ is in $X_\G$ if it satisfies one of the following conditions: 
\begin{itemize}
	\item $v_i\in A$ for all $i$. 
	\item $v_i\in B$ for all $i$. 
	\item $v_1\in A, v_2\in B$ and $v_1$ and $v_2$ are connected by an edge in $\G$. 
	\item $v_2\in A, v_3\in B$ and $v_2$ and $v_3$ are connected by an edge in $\G$. 
	\item $v_3\in A, v_1\in B$ and $v_3$ and $v_1$ are connected by an edge in $\G$. 
\end{itemize}

This construction is the same as the one in \cite{Lod-18,Kro-21}. Thus, we immediately obtain hyperbolicity of this cube complex.

The next step is to define a homomorphism to $\Z^n$. We do this by defining $n$ continuous functions $f_i: X_\G \to S^1$ which induce linearly independent characters on fundamental groups. To define them, we first define maps $g_i\colon A\ast B\to S^1$, then extend affinely to cubes in $\mathbb{X}$ and finally restrict to $X_{\Gamma}$. To define the $g_i$, we parametrise each edge of $A*B$ by $[0,1]$ identifying $0$ with the end point in $A$ and $1$ with the end point in $B$. For each $i\in \{1, \dots, n\}$, the map $g_i\colon A*B\to S^1$ is then the map that sends the vertex set $A\sqcup B$ to the unique vertex of $S^1$ and extends to edges $e$ as follows: 
\begin{itemize}
	\item $g_i(t) = t$ if $e$ is an edge between a vertex in $A_i^+$ and a vertex in $B_j^+$ for some $j$. 
	\item $g_i(t) = t$ if $e$ is an edge between a vertex in $A_j^+$ and a vertex in $B_i^+$ for some $j$.
	\item $g_i(t) = t$ if $e$ is an edge between a vertex in $A_i^-$ and a vertex in $B_j^-$ for some $j$. 
	\item $g_i(t) = t$ if $e$ is an edge between a vertex in $A_j^-$ and a vertex in $B_i^-$ for some $j$. 
	\item $g_i(t) = 1-t$ if $e$ is an edge between a vertex $A_i^+$ and a vertex in $B_j^-$ for some $j$.
	\item $g_i(t) = 1-t$ if $e$ is an edge between a vertex $A_j^+$ and a vertex in $B_i^-$ for some $j$.
	\item $g_i(t) = 1-t$ if $e$ is an edge between a vertex $A_i^-$ and a vertex in $B_j^+$ for some $j$.
	\item $g_i(t) = 1-t$ if $e$ is an edge between a vertex $A_j^-$ and a vertex in $B_i^+$ for some $j$.
	\item $g_i(t) = 0$ otherwise. 
\end{itemize}

Extending the $g_i$ affinely to cubes we get a map $(A*B)\times (A*B)\times (A*B)\to S^1$ that restricts to $g_i$ on every factor. The map $f_i:X_{\G}\to S^1$ is then the restriction of this map to the subcomplex $X_{\G}\subset \mathbb{X}$. The induced homomorphism $\phi_i=f_{i\ast}\colon \pi_1(X_\G)\to \Z$ has image of index 2 if $n=1$ and is surjective if $n>1$. The subset $\{\phi_1, \dots, \phi_n\}\subset {\rm Hom}(\pi_1(X_\G), \R)$ is linearly independent; to see this, consider the restriction of the $f_j$ to a subgraph of the form $A_i\ast B_i$ contained in one of the factors, where $A_i = A_i^-\sqcup A_i^+, B_i = B_i^-\sqcup B_i^+$, and observe that this restriction is non-trivial on fundamental groups if and only if $i=j$. By linear independence of the $\phi_i$ the image of the homomorphism $\phi=\left(\phi_1,\dots,\phi_n\right)\colon \pi_1(X)\to \Z^n$ has finite index in $\Z^n$.

\begin{theorem}
	The kernel $\ker(\phi)$ is a finitely presented subgroup of $\pi_1(X_\G)$ which is not of type $F_3$. 
\end{theorem}
\begin{proof}
	By Theorem \ref{thm:Sigma-BNSR} it suffices to show that $S(\pi_1(X_\G),\ker(\phi))\subseteq \Sigma^2(\pi_1(X_\G))$ to deduce that $\ker(\phi)$ is finitely presented. We will use Theorem \ref{thm:Main-BNSR} to prove this. First observe that every element of $S(\pi_1(X_\G),\ker(\phi))$ has a representative of the form $\sum_{i=1}^{n}\lambda_i \phi_i:\pi_1(X_\G)\to \mathbb{R}$ with $\left(\lambda_1,\dots,\lambda_{n}\right)\in \mathbb{R}^{n}\setminus \left\{0\right\}$ and, conversely, any such homomorphism defines an element of $S(\pi_1(X_\G),\ker(\phi))$. 
	
	Let $\left[\psi\right]=\left[\sum_{i=1}^{n}\lambda_i \phi_i\right]\in S(\pi_1(X_\G),\ker(\phi))$. We may assume that $\lambda_1> 0$; the other cases being analogous. We define a $\psi$-equivariant height map $h\colon \widetilde{X_\G}\to \mathbb{R}$ by $h = \sum_{i=1}^{n}\lambda_i h_i$, where the $h_i\colon \widetilde{X_\G}\to \mathbb{R}$ are lifts of the $f_i$ to $\widetilde{X_\G}$
	
	To apply Theorem \ref{thm:Main-BNSR} we need to determine the connectivity properties of all ascending and descending living links of all dead simplices in $L_{v,\widetilde{X_\G}}$ for every vertex $v=(v_1,v_2,v_3)$ of $\widetilde{X_\G}$. We will distinguish two cases:
	\begin{itemize}[leftmargin=1.5cm]
		\item[{\bf Case 1}] either all $v_i$ are in $A$, or all $v_i$ are in $B$;
		\item[{\bf Case 2}] there are $v_i$'s in both $A$ and $B$. 
	\end{itemize}
	\vspace{.2cm}
	
{\noindent \bf Case 1.} We assume that $v_i\in A$ for all $i$. The case where $v_i\in B$ for all $i$ can be treated similarly. 
	
	We start by determining the link $L_{v,\widetilde{X_\G}}$ of $v$. The vertices adjacent to $v$ in $\mathbb{X}$ are obtained by changing one coordinate $v_i$ to $w_i$ for some $w_i\in B$. Such a vertex is in the complex $X_\G$ if and only if $w_i$ is connected to $v_{i-1}$ in $\G$ where indices are taken modulo 3. By computing what happens when we change two of the vertices we see that the link of $v$ is a join of three discrete sets corresponding to the vertices adjacent to $v_1, v_2$ and $v_3$ in $\G$. 
	
	We check that the connectivity properties of all ascending and descending links of dead simplices in $L_{v,\widetilde{X}_{\G}}$ satisfy the conditions of Theorem \ref{thm:Main-BNSR}, beginning with the case that $\sigma$ is the empty (dead) simplex, where we need to show that they are simply connected. In this case, the ascending (resp. descending) links are the full subcomplexes of $L_{v,\widetilde{X_\G}}$ spanned by vertices corresponding to edges  from $v$ to $w$ with $h(w)>h(v)$ (resp. $h(w)<h(v)$). It suffices to compute the ascending living link for a given height function as these will be the descending links for $-h$. 
	
	Since we are working in a join, simple connectivity of the ascending living link of $\sigma$ is equivalent to showing that there is a vertex corresponding to an edge where $h$ is ascending in each of the three sets in the join. Indeed, the full subcomplex spanned by these three non-empty sets will then be simply connected. 
	We begin with adjacent vertices of the form $(v_1, w_2, v_3)$. 
	They are in the ascending living link exactly when $\sum_{i=1}^n\lambda_i(h_i(v_2)-h_i(w_2)) < 0$. 
	Suppose that $v_2$ is in $A_j^+$; the case that $v_2\in A_j^-$ is the same with signs switched. 
	We divide into three cases, depending whether $\lambda_j > 0$, $\lambda_j< 0$, or $\lambda_j = 0$. 
	
	First assume $\lambda_j>0$. Then there is an adjacent vertex $(v_1, w_2, v_3)$ where $w_2$ is in $B_j^+$. For such a vertex we have $h_j(v_2) - h_j(w_2) = -1$ and $h_i(v_2) - h_i(w_2) = 0$ if $i\neq j$. 
	We deduce that $h(v) - h(w) = -\lambda_j < 0$. 

	If $\lambda_j < 0$ pick $w_2$ in $B_j^-$ and argue similarly.
	
	In the case $\lambda_j = 0$, we take an adjacent vertex of the form $(v_1, w_2, v_3)$, where $w_2\in B_1^+$. 
	Then $h_1(v_2) - h_1(w_2) = -1 = h_j(v_2) - h_j(w_2)$ and $h_i(v_2) = h_i(w_2)$ for all other $i$. 
	Hence we see that $h(v) - h(w) = \sum_{i=1}^n\lambda_i(h_i(v_2)-h_i(w_2)) = -\lambda_1 <0.$ 
	
	We conclude that there is a vertex in each of the three sets that make up the join so the ascending living link is simply connected. 
	Similarly, the descending living link is simply connected. 
	
	This also shows that the ascending and descending living links of dead vertices are connected. 
	To see this note that the link of each vertex in $L_{v, X_\G}$ is a join of two discrete sets. 
	The preceding paragraphs show that there is an ascending living vertex in each of these sets. 
	Thus the ascending living link of a vertex in $L_{v, X_\G}$ is a join of two non-empty sets and therefore connected. 
	Similarly, the ascending living link of an edge in $L_{v, X_\G}$ is non-empty. \vspace{.3cm}
	
	{\noindent \bf Case 2.}	We now consider the case of a vertex of the form $x = (x_1, x_2, x_3)$ where $x_1\in A, x_2\in B, x_3\in A$. The other cases can be treated similarly. 
	
	Since $x$ is a vertex of $X_\G$ we see that there is an edge from $x_1$ to $x_2$ in $\G$. 
	We first study the link of such a vertex. 
	Vertices of the form $y = (x_1, x_2, y_3)$ with $y_3\in B$ are in the complex $X_\G$ since there is an edge in $\G$ from $x_1$ to $x_2$. 
	Similarly vertices of the form $y = (x_1, y_2, x_3)$ with $y_2\in A$ are in the complex $X_\G$ since $x_1, y_2, x_3\in A$. 
	Finally vertices of the form $y = (y_1, x_2, x_3)$ with $y_1\in B$ are in the complex exactly when there is an edge from $y_1$ to $x_3$ in $\G$. 
	Thus the vertex set of $L_{x, X_\G}$ can naturally be identified with $A\sqcup B\sqcup Lk(x_3, \G)$.
	
	There is an edge in the link between vertices corresponding to $(x_1, x_2, y_3)$ and $(x_1, y_2, x_3)$ exactly when the vertex $(x_1, y_2, y_3)$ is in $X_\G$. 
	The latter is the case precisely when there is an edge from $y_2$ to $y_3$ in $\G$. 
	Moreover, there is always an edge between vertices corresponding to $(x_1, y_2, x_3)$ and $(y_1, x_2, x_3)$. Indeed, if $(y_1, x_2, x_3)$ is in the complex, then there is an edge between $x_3$ and $y_1$ in $\G$ which implies that  $(y_1, y_2, x_3)$ is also in the complex. 
	Similarly, there is an edge between all vertices corresponding to $(x_1, x_2, y_3)$ and $(y_1, x_2, x_3)$, because $y_1, x_2, y_3\in B$ implies that $(y_1, x_2, y_3)$ is in the complex. 
	Thus $L_{x, X_\G}$ is a join $\G\ast Lk(x_3, \G)$. 
	
	We now compute the ascending and descending living link of $x$. 
	We will only provide the computation for the ascending living link, since the computation for the descending living link is analogous.
	Since the link is a join, the ascending living link splits as $\G_1\ast V$, where $\G_1$ is a subgraph of $\G$ and $V$ is a subset of $Lk(x_3, \G)$. 
	To show that it is simply connected, it suffices to show that $\G_1$ is connected and that $V$ is non-empty. 
	The analysis from Case 1 shows that $V$ is non-empty, so we are left with computing $\G_1$. 
	Let $\mathcal{V}$ be the vertex set of $\G_1$. 
	Since $\G_1$ is the full subgraph of $\G$ spanned by $\mathcal{V}$, it suffices to determine $\mathcal{V}$. 
	
	We begin by checking which vertices of the form $y = (x_1, x_2, y_3)$ are in $\mathcal{V}$, by computing for which $y_3\in B$ we have $h(x) - h(y) > 0$. 
	Suppose that $x_3\in A_j^+$. The case $x_3\in A_j^-$ is similar with signs switched. 
	As above, we split into the cases $\lambda_j>0, \lambda_j<0$, or $\lambda_j = 0$. 
	
	First assume $\lambda_j>0$. In this case there is an adjacent vertex $(x_1, x_2, y_3)$, where $y_3$ is in $B_j^+$. For such a vertex we have $h_j(x_3) - h_j(y_3) = -1$ and $h_i(x_3) - h_i(y_3) = 0$ if $i\neq j$. We deduce that $h(x) - h(y) = -\lambda_j < 0$. 
	
	If $\lambda_j < 0$ pick $y_3$ in $B_j^-$.
	
	In the case $\lambda_j = 0$, we take an adjacent vertex of the form $(x_1, x_2, y_3)$ where $y_3\in B_1^+$. 
	Thus $h_1(x_3) - h_1(y_3) = -1 = h_j(x_3) - h_j(y_3)$ and $h_i(x_3) = h_i(y_3)$ for all other $i$. 
	We deduce that $h(x) - h(y) = \sum_{i=1}^n\lambda_i(h_i(x_3)-h_i(y_3)) = -\lambda_1 <0.$ 
	
	Hence $\mathcal{V}$ contains at least one of the sets $B_j^+, B_j^-, $ or $B_1^+$. 
	Since the functions $h_i$ only depend on the set $B_k^s$ containing a vertex and not on the vertex itself, we deduce that $\mathcal{V}\cap B$ is a non-empty disjoint union of sets of the form $B_k^s$. 
	
	Repeating this analysis with vertices of the form $y = (x_1, y_2, x_3)$, we obtain that $\mathcal{V}\cap A$ is a non-empty disjoint union of sets of the form $A_l^t$. 
	
	By assumption, for each choice $k, l\in \{1, \dots, n\}$ and $s, t\in \{-, +\}$ the subgraph spanned by $A_l^t\cup B_k^s$ is connected. 
	Thus, the subgraph spanned by $\mathcal{V}$ can be decomposed as a union of connected graphs, each having non-empty intersection with one of the previous graphs. 
	In particular, $\G_1$ is connected, completing the proof that the ascending living link of $x$ is connected.
	
	Finally, we study the ascending living links of dead simplices in $L_{x, X_\G}$. 
	Since these are invariants of the corresponding cube, we have already dealt with the case of dead edges in $L_{x, X_\G}$. Indeed, every square contains a vertex of the form $(v_1, v_2, v_3)$, where $v_i\in A$ for all $i$ or $v_i\in B$ for all $i$ and we have already treated edges ending in such a vertex in Case 1. The same reasoning applies to dead vertices of the form $(x_1,y_2,x_3)$.
	Thus we are left with treating dead vertices in the link corresponding to edges from $x$ to vertices of the form $(y_1, x_2, x_3)$ or $(x_1, x_2, y_3)$. 
	
	We begin with a vertex of the form $y = (x_1, x_2, y_3)$. 
	Let $\sigma$ be the corresponding vertex of $L_{x, X_\G}$. 
	Since $L_{x, X_\G}$ is a join of the form $\G\ast Lk(x_3,\Gamma)$, $\sigma$ is a vertex of $\G$ corresponding to an element of $B$. 
	Thus the link of $\sigma$ in $L_{x, X_\G}$ is of the form $W\ast V$ where $W=Lk(\sigma,\Gamma)$ and $V=Lk(x_3,\Gamma)$.  
	Since all the subgraphs spanned by $A_k^s\cup B_l^t$ are connected we see that $W\cap A_k^s$ and $V\cap B_k^s$ are non-empty for all choices of $k, s$. 
	Using the same arguments as before, we deduce that the ascending living link of $\sigma$ is also a join of two non-empty sets and thus connected. 
	
	For a vertex $\tau$ corresponding to $(y_1, x_2, x_3)$ we have that the link of $\tau$ in $L_{x, X_\G}$ is exactly $\G$. 
	By considering again adjacent vertices $z$ of $x$ coming from $\G$ which satisfy $h(z)>h(x)$, we see that the ascending living link of $\tau$ is a subgraph of $\G$ spanned by sets of the form $A_k^s$ and $B_l^t$. 
	The same analysis as before shows that the vertex set of this graph must contain at least one set of the form $A_k^s$ and one set of the form $B_l^t$. 
	Thus, we see that this subgraph can be decomposed as a union of connected graphs, each of which has non-empty intersection with one of the previous graphs. 
	Thus, it is connected, which completes the proof. 
	
	To see that the kernel is not of type $F_3$, we use the Euler characteristic in the same way as before. The calculations in \cite[Proposition 4.3]{Kro-21} show that in this case the Euler characteristic of $X_\G$ is $2p^3(1-36n^2 + 192n^4 -256n^6) + 48p^2(n^2 - 4n^4)$, which is always negative for $n\geq 1$. 
\end{proof}

\bibliography{References}
\bibliographystyle{amsplain}

\end{document}